\documentclass[12pt,cd]{amsart}
\usepackage{latexsym}
\usepackage{amsmath,amsfonts,amssymb,amsthm,mathtools,amscd,bm,mathabx}
\usepackage{graphics}
\usepackage{csquotes}
\usepackage{blindtext}
\usepackage{xcolor}

 2

\usepackage{hyperref}
\newtheorem{theorem}{Theorem}[section]
\newtheorem{prop}[theorem]{Proposition}

\newtheorem{cor}[theorem]{Corollary}
\newtheorem{definition}[theorem]{Definition}

\newtheorem{remark}[theorem]{Remark}
\numberwithin{equation}{section}

\title[Superposition of Harmonic Surfaces]{Superposition of Harmonic Surfaces: Helical Motifs in Lamellar Structures}
\author{Priyank Vasu}
\address{Department of Mathematics, Indian Institute of Technology Patna, Bihta, Patna-801106, Bihar, India}

\thanks{Priyank Vasu: \href{mailto:priyank_2121ma16@iitp.ac.in}{priyank\_2121ma16@iitp.ac.in}}


\begin{document}
\maketitle
\begin{abstract}
    We study harmonic surfaces in $\mathbb{R}^3$ through the framework of harmonic Enneper immersions and prove a superposition principle for such surfaces. We prove that minimal and maximal surfaces admit a decomposition into harmonic components. Applications include the construction of finite and infinite configurations of helical motifs, an asymptotic analysis via multipole expansions, and the modelling of twist grain boundary phases in lamellar systems.
\end{abstract}

%==================================================================================
\section{Introduction}

An immersion $X: \Omega \to \mathbb{R}^3$ is called harmonic if its coordinate functions are harmonic, i.e.\ $\Delta X = 0$. Minimal surfaces form a distinguished subclass of harmonic immersions corresponding to the conformal case, where the Weierstrass representation provides a harmonic parametrisation. In general, however, harmonic immersions need not be conformal (see \eqref{example}), and thus extend the classical theory of minimal surfaces. The study of harmonic immersions in $\mathbb{R}^3$ dates back to the work of Klotz \cite{Klotz1967,Klotz1968,KlotzMilnor1979,KlotzMilnor1980}. Recently, harmonic immersions have been studied in the context of quasiconformal mappings, Gauss map properties, and Weierstrass-type representations \cite{AlarconLopez2013,Kalaj2013,Vasu2026}. Superposition principles for minimal and maximal surfaces have been investigated in \cite{Dorff2012} and \cite{Paul2024}, respectively; in this work, we extend these ideas to the broader class of harmonic immersions.
 
Harmonic surfaces arise naturally in a wide range of physical and biological systems. Recent studies have revealed the presence of helical motifs within the lamellar structures of the endoplasmic reticulum,  plant photosynthetic membrane and cyanobacterial thylakoids \cite{Terasaki2013, Mustardy2008, Liberton2011, Bussi2019}. Twist grain boundaries phases, composed of infinitely many helical motifs of a single handedness arranged along a line, exact theoretical descriptions and are well understood \cite{Kamien1999, Santangelo2006, Matsumoto2017}.

Due to their close connection with minimal surface theory, such problems have traditionally been studied primarily from the viewpoint of minimal surfaces \cite{Powers2002, Kamien1999, daSilva2021}. Santangelo and Kamien introduced a different perspective in \cite{Santangelo2007}, in which they investigated twist grain boundaries in liquid crystals using harmonic surfaces rather than minimal ones. In particular, Kamien \cite{Kamien2001} showed, using a Ramanujan identity, that an infinite superposition of harmonic graphs yields a Scherk's first surface, a minimal and harmonic surface.

Motivated by this approach, we study harmonic surfaces via their Enneper representations. Exploiting the linearity of harmonic functions, we derive a superposition principle for harmonic immersions, as well as a superposition result for harmonic graphs. We further establish a connection between harmonic surfaces and minimal or maximal surfaces that can be decomposed into a superposition of two harmonic immersions, and illustrate several applications arising in geometric and physical contexts.

\section{Harmonic Enneper immersions}
\label{Sect::EnneperImmersions}

%=============================================================
Let $\Omega \subset \mathbb{C}$ be a connected open set. We identify $\mathbb{R}^2$ with $\mathbb{C}$ via $z = x + iy$, where $(x,y) \in \mathbb{R}^2$, and denote by $(r,\theta)$ the polar coordinates in $\mathbb{C}$.

\begin{definition}
A map $X = (X_1, X_2, X_3) : \Omega \to \mathbb{R}^3$ is called a \emph{harmonic immersion}
if:
\begin{enumerate}
    \item $X$ is an immersion,
    \item each component function $X_j$, $j=1,2,3$, is harmonic on $\Omega$.
\end{enumerate}
A subset $\mathcal{S} \subset \mathbb{R}^3$ is called a \emph{harmonic surface} if there exists
a harmonic immersion $X : \Omega \to \mathbb{R}^3$ such that $\mathcal{S} = X(\Omega)$.
In this case, $X$ is called a \emph{harmonic parametrisation} of $\mathcal{S}$.
\end{definition}

A harmonic immersion may admit multiple harmonic parametrisations. For example,
\begin{equation} \label{example}
\begin{aligned}
    Y_1 &: \mathbb{C} \to \mathbb{R}^3, 
    &Y_1(z) &= \operatorname{Re}\big(e^z,\, z e^z,\, i z \big), \\
    Y_2 &: \{z \in \mathbb{C} : \operatorname{Re}(z) > 0\} \to \mathbb{R}^3, 
    &Y_2(z) &= \operatorname{Re}\big(\sinh z,\, i \cosh z,\, i z \big),
\end{aligned}
\end{equation}
define two distinct harmonic parametrisations of the same surface, namely a half-helicoid.

We recall the Wirtinger operators $\partial_z := \frac{1}{2}\left(\frac{\partial}{\partial x} - i \frac{\partial}{\partial y}\right),
\,
\partial_{\bar z} := \frac{1}{2}\left(\frac{\partial}{\partial x} + i \frac{\partial}{\partial y}\right),$ where $z = x + iy \in \Omega$. We also use the notation $X_x = \frac{\partial X}{\partial x}, 
X_y = \frac{\partial X}{\partial y}, 
X_z = \partial_z X, 
X_{\bar z} = \partial_{\bar z} X.$ Let $X = (X_1, X_2, X_3)$ be a harmonic immersion. We define the complex derivative vector
\[
\Phi = (\phi_1, \phi_2, \phi_3) := (\partial_z X_1, \partial_z X_2, \partial_z X_3).
\]
%=============================================================

\begin{definition}\label{definition Weierstrass representation} \cite{AlarconLopez2013}
The pair $(\Omega, \Phi)$ is called the Weierstrass representation of $X$. 
The holomorphic coefficient of the quadratic differential $\mathfrak{H}:= \sum_{j=1}^3 \phi_j^2$ on $\Omega$ is called the Hopf differential of the harmonic immersion $X$.
\end{definition}

\begin{remark}
If the Hopf differential satisfies $\mathfrak{H} \equiv 0$ on $\Omega$, then the corresponding harmonic immersion $X$ is minimal. 
Notice that $Y_2$ is minimal, whereas $Y_1$ is not in \eqref{example}.
\end{remark}

\begin{theorem}\cite{Duren2004}
If $f = u + iv$ is a harmonic function on $\Omega$, then it can be expressed as $f = L + \overline{P} $, where $L$ and $P$ are holomorphic functions on $\Omega$.
\end{theorem}

An orientation-preserving smooth mapping \( k: \Omega \to \Omega' \) between two domains in \( \mathbb{C} \) is called \textit{quasiconformal} if \( |\nu_k| < 1 \), where \( \nu_k \) is called the analytic dilatation (or the second Beltrami coefficient) of  \( k \), given by $\nu_k:= \frac{\bar{k}_{\bar{z}}}{ k_z}$ (see \cite{Duren2004}).

%==========================================================================================

In general, it is not easy to determine when a harmonic map 
\( X : \Omega \to \mathbb{R}^3 \) is an immersion. This constitutes one of the main difficulties in constructing such maps. To address this issue, we introduce the notion of \emph{harmonic Enneper immersions}, as presented in \cite{Vasu2026}.

First, we identify $\mathbb{R}^3 \equiv \mathbb{C} \times \mathbb{R}$. Now, let \( h : \Omega \to \mathbb{R} \) and \( f : \Omega \to \mathbb{C} \) be harmonic functions. Then \( f \) admits a decomposition of the form 
\[
f(z) = L(z) + \overline{P(z)},
\]
where \( L \) and \( P \) are holomorphic functions on \( \Omega \). We define a map
\begin{equation}\label{Eq::EnneperGraph}
    X(z) := (f(z), h(z)) = \bigl(L(z) + \overline{P(z)},\, h(z)\bigr), \quad z \in \Omega.
\end{equation}
The map \( X \) is a harmonic immersion with Hopf differential \( \mathfrak{H} \) provided that \( L \) and \( P \) satisfy
\begin{equation}\label{Eq: immersion}
    L_z P_z + (h_z)^2 = \mathfrak{H}, 
    \qquad \text{and} \qquad 
    |L_z| \ne |P_z|.
\end{equation}
The second condition ensures that \( X \) is an immersion. We refer to such a map $X = (L + \overline{P}, h)$ as a \emph{harmonic Enneper immersion}. Conversely, every harmonic immersion can be locally represented in this form. This representation simplifies the verification of the immersion condition. See \cite{Vasu2026} for a proof.
%===================================================================

For harmonic functions, since the Laplacian is linear on $C^2(\Omega)$ (the space of twice continuously differentiable functions on $\Omega$), any linear combination of harmonic functions on $\Omega$ is again harmonic. However, a linear combination of harmonic immersions is not, in general, a harmonic immersion. For harmonic Enneper immersions, we have the following result.

\begin{theorem}(Superposition principle)\label{Superposition principle}
For each $i=1,\dots,n$, let $X_i=(f_i,h_i):\Omega \to \mathbb{C} \times \mathbb{R}$ be a harmonic Enneper immersion, where $f_i:\Omega\to\mathbb{C}$ and $h_i:\Omega\to\mathbb{R}$ are harmonic functions. Assume that all $f_i$ have the same complex dilatation $\nu=\frac{\bar{(f_i)}_{\bar{z}}}{ (f_i)_z}$, and  let $a_i, b_i \in \mathbb{R}$ be constants such that  $\sum_{i=1}^n a_i f_i \neq 0$. Then the map $X:\Omega \to \mathbb{C} \times \mathbb{R}$ defined by
\[
X = \sum_{i=1}^n X_i
:= \left(\sum_{i=1}^n a_i f_i,\ \sum_{i=1}^n b_i h_i\right)
\]
is a harmonic Enneper immersion on $\Omega$. Moreover, writing $f_i = L_i + \overline{P_i}$ with $L_i, P_i$ holomorphic, we have
\[
X = \left(\sum_{i=1}^n a_i L_i + \sum_{i=1}^n a_i \overline{P_i},\ \sum_{i=1}^n b_i h_i\right).
\]
\end{theorem}

%===================================================================

\begin{proof}
For each $i=1,\dots,n$, let $X_i=(f_i,h_i)$ be a harmonic Enneper immersion, where $f_i=L_i+\overline{P_i}$ with $L_i$ and $P_i$ holomorphic on $\Omega$. The projection of $X_i$ onto the \(xy\)-plane is the harmonic mapping \(f_i\), and its complex dilatation is $\nu_{f_i}=\nu \text{ for all } i=1,\dots,n.$
Now define
\[
f:=\sum_{i=1}^n a_i f_i
\]
with \(a_i\in \mathbb{R}\). Then
\[
f=\sum_{i=1}^n a_iL_i+\overline{\sum_{i=1}^n a_iP_i}.
\]
Set
\[
L:=\sum_{i=1}^n a_iL_i, \qquad P:=\sum_{i=1}^n a_iP_i.
\]
Since each \(L_i\) and \(P_i\) is holomorphic on \(\Omega\), it follows that \(L\) and \(P\) are also holomorphic on \(\Omega\), and hence
$f=L+\overline{P}$
is harmonic function on \(\Omega\). Moreover, since \(\nu_{f_i}=\nu\) for every \(i\), we have $(P_i)_z=\nu (L_i)_z, \, i=1,\dots,n.$ Therefore,
\[
P_z=\sum_{i=1}^n a_i (P_i)_z
   =\sum_{i=1}^n a_i \nu (L_i)_z
   =\nu \sum_{i=1}^n a_i (L_i)_z
   =\nu L_z.
\]
Thus the complex dilatation of \(f\) is $\nu_f=\frac{P_z}{L_z}=\nu.$ Hence, the planar part of
\[
X=\left(\sum_{i=1}^n a_i f_i,\ \sum_{i=1}^n b_i h_i\right)
\]
has the same complex dilatation \(\nu\). Since the vertical component \(\sum_{i=1}^n b_i h_i\) is a harmonic function as a linear combination of harmonic functions, the map \(X\) satisfies the same conditions as in \eqref{Eq: immersion}. Consequently, \(X\) is a harmonic Enneper immersion on \(\Omega\).
\end{proof}

%===================================================================

\begin{remark}
If we choose $b_i = a_i$ for all $i$, then this shows that harmonic Enneper immersions form a real vector space, provided they have the same complex dilatation. It is also important to note that there exist examples where the complex dilatations satisfy $\nu_{f_j} \neq \nu_{f_k}$ for some harmonic immersion $X_j$ and $X_k$, yet their sum $X = X_j + X_k$ is still a harmonic immersion; see, for instance, Theorem \ref{minimal and maximal}.
\end{remark}

 \begin{theorem}\cite{Lewy1936}\label{theorem univalent}
The harmonic function $f=L+\overline{P}$ is locally univalent and orientation preserving on $\Omega$ if and only if $|\nu_f(z)|<1,\text{for all } z\in\Omega.$
\end{theorem}
A surface of the form $X = \{(x,y,\phi(x,y)) \mid (x,y) \in \Omega\}$ is called a graph over $\Omega$. We call an immersion $X = (f,h)$ a \emph{harmonic graph} if $f: \Omega \to \mathbb{C}$ is an orientation-preserving harmonic univalent mapping. One advantage of this representation is that the univalence of the harmonic function $f$ ensures that the associated harmonic Enneper immersion defines a graph over the image of $f(\Omega)$.
It follows immediately from the Theorem \cite[Proposition 3.3]{Vasu2026} and the above Theorem \ref{theorem univalent} that the following result holds.

\begin{prop}
Let $f = L + \overline{P}: \Omega \to \mathbb{C}$ be an orientation-preserving harmonic univalent mapping of a domain $\Omega$ onto a domain $D \subset \mathbb{C}$. Suppose that $h: \Omega \to \mathbb{R}$ is a harmonic function. Then the map $X: \Omega \to \mathbb{C} \times \mathbb{R}$ defined by $X(z) = (L(z) + \overline{P(z)},\, h(z))$ defines a harmonic graph.
\end{prop}

%===================================================================

%===================================================================
As a direct consequence of the Superposition principle in Theorem \ref{Superposition principle} and the previous proposition, we obtain the following superposition result for harmonic graphs.

\begin{cor}
For each $i=1,\dots,n$, let $f_i: \Omega \to \mathbb{C}$ be orientation-preserving harmonic univalent mappings with the same complex dilatation, and let $h_i: \Omega \to \mathbb{R}$ be harmonic functions. Let $a_i, b_i \in \mathbb{R}$ be constants such that  $\sum_{i=1}^n a_i f_i \neq 0$. Then the map
\[
X=\sum_{i=1}^n X_i
:=\left(\sum_{i=1}^n a_i f_i,\ \sum_{i=1}^n b_i h_i\right)
\]
defines a harmonic graph on $\Omega$.
\end{cor}
It is worth noting that the map $X = \left( z,\ \sum_{i=1}^n b_i h_i \right)$
defines a harmonic graph on $\Omega$, where each $h_i : \Omega \to \mathbb{R}$ is harmonic function.

Minimal and maximal surfaces are closely related to harmonic immersions, as their coordinate functions are harmonic in conformal parametrisations. We briefly recall their Weierstrass representations. A minimal surface in $\mathbb{R}^3$ can locally be represented in terms of holomorphic data $(F, G)$, where $F$ is a holomorphic function and $G$ is a meromorphic function, via the Weierstrass-Enneper representation, as in equation \eqref{minimal decomposition} (see \cite{Osserman1986} for details). Similarly, maximal surfaces in the Lorentz--Minkowski space $\mathbb{L}^3$ admit an analogous representation in terms of holomorphic data (see \cite{Kobayashi1983}). 

The following result shows that minimal and maximal surfaces admit a decomposition in terms of harmonic immersions.

\begin{theorem}\label{minimal and maximal}
A minimal surface in $\mathbb{R}^3$ or a maximal surface in $\mathbb{L}^3$ can be written as a superposition of two harmonic immersions in $\mathbb{R}^3$.
\end{theorem}
\begin{proof}

 Let $X_{min}$ be a minimal surface given by the Weierstrass data $(F, G)$. Then
\begin{equation}\label{minimal decomposition}
\begin{aligned}
X_{min} &= \operatorname{Re} \int^z 
\begin{pmatrix}
1-G^2 \\
-i(1+G^2) \\
2G
\end{pmatrix}
F\,dz \\
&= \operatorname{Re} \int^z 
\begin{pmatrix}
1 \\
-i \\
G
\end{pmatrix}
F\,dz
+ \operatorname{Re} \int^z 
\begin{pmatrix}
-1 \\
-i \\
\frac{1}{G}
\end{pmatrix}
F G^2\,dz \\
&= \operatorname{Re} \int^z 
\begin{pmatrix}
1 \\
-i \\
G
\end{pmatrix}
F\,dz
+ A\, \operatorname{Re} \int^z 
\begin{pmatrix}
1 \\
-i \\
\frac{1}{G}
\end{pmatrix}
F G^2\,dz
\end{aligned}
\end{equation}
where $A = \mathrm{diag}(-1,1,1)$. Let $X_1$ and $A X_2$ denote the first and second terms in \eqref{minimal decomposition}, respectively, so that $X_{min} = X_1 + A X_2.$ In the harmonic Enneper representation of $X_1$, we have $(L_z)_1 = F$ and $(P_z)_1 = 0$. Similarly, for $X_2$, we have $(L_z)_2 = 0$ and $(P_z)_2 = -F G^2$. These satisfy \eqref{Eq: immersion}; hence, $X_1$ and $X_2$ are harmonic immersions.

We can similarly show that, for a maximal surface $X_{max}$ with Weierstrass data $(F,G)$, we have $X_{max} = X_1 + B X_2,$ where $B = \mathrm{diag}(1,-1,1)$. Hence the proof.
    
\end{proof}

In the following sections, we explore applications of the above framework, showing how the superposition of harmonic surfaces gives rise to a variety of geometrically and physically significant surfaces.

\section{Harmonic Graphs via Analytic Functions} 
A harmonic graph \( X(z) = (z, h(z)) \) can be naturally generated from the real or imaginary part of an analytic function, where \( h(z) \) denotes either the real or imaginary part of an analytic function defined on a domain \( \Omega \subset \mathbb{C} \).

As a simple example, consider the analytic function \( \log z \), defined on a simply connected domain \( \Omega \subset \mathbb{C} \setminus \{0\} \) with a fixed branch. Then
\[
\log z = \ln|z| + i \arg z,
\]
and both \( \ln|z| \) and \( \arg z \) are harmonic functions on \( \Omega \). Hence, for real constants \( a, b \in \mathbb{R} \), we may define a harmonic Enneper immersion by taking \( L(z) = z \), \( P(z) = 0 \), and \( h(z) = a + b \ln|z| \), which yields
\begin{equation}
    X(z) = \bigl( z,\ a + b \ln|z| \bigr).
\end{equation}

This equation describes the membrane shape in the outer region, away from localised forces, where bending effects become significant (see \cite{Powers2002}).
Another fundamental example is given by
\begin{equation}\label{eq::SingleHelicoid}
X(z) = \bigl(z, a \arg(z - z_0)\bigr),
\end{equation}
which represents a helicoid of pitch \( a \) centred at \( z_0 \). It is well known that helicoids and planes are the only minimal surfaces that can be expressed as graphs of harmonic functions. 

A key advantage of this representation is that, due to the linearity of the Laplacian, harmonic functions can be superposed, allowing the construction of more complex surfaces by combining simpler building blocks, which we refer to as \emph{helical motifs}.

\section{Superposition of Finitely Many Helical Motifs}

We begin by analysing the superposition of two helical motifs. Let two helicoids of pitches \( a_1 \) and \( a_2 \) be centered at points \( z_1 \) and \( z_2 \), respectively, with separation \( R = |z_2 - z_1| > 0 \). The associated harmonic function of the harmonic graph is
\[
h(z) = a_1 \arg(z - z_1) + a_2 \arg(z - z_2).
\]
To obtain a periodic layered structure, the pitches must satisfy \( a_1 = a_2 \). For simplicity, we first consider the symmetric case \( a_1 = a_2 = a \). At large distances, the resulting surface is well approximated by two intertwined helicoids of pitch \( 2a \), displaced along their axes. A particularly interesting configuration arises in the dipole case, where the pitches are opposite, namely \( +a \) and \( -a \), with centres located at \( z_1 = \tfrac{R}{2} \) and \( z_2 = -\tfrac{R}{2} \). This regime is relevant in physical applications such as those discussed in \cite{Terasaki2013, Bussi2019}.

More generally, for \( n \) helical motifs with pitches \( a_1, \dots, a_n \) located at \( z_1, \dots, z_n \), we consider the harmonic function
\begin{equation}\label{eq::hOfManyMotifs}
    h(z) = \sum_{k=1}^n a_k \,\mathrm{Im}[\ln(z - z_k)] = \sum_{k=1}^n a_k \arg(z - z_k).
\end{equation}

\section{The multipole expansion in two-dimensional electrostatics}
\label{Sect::MultipoleExpansion}

The construction of a harmonic graph associated with a harmonic function whose singularities can be interpreted, by analogy with electrostatics, as charges. This suggests using a multipole expansion in which helical motifs serve as point charges. In two-dimensional electrostatics, the potential of a point charge has a logarithmic singularity
\[
V(r,\theta)=V(r)=\ln\frac{1}{r},
\]
whereas in our setting, the appropriate singularity is angular,
\[
h(r,\theta)=h(\theta)=\theta.
\]
Up to constants, $h=\theta$ is the only purely angular harmonic function in two dimensions, just as $h=\ln r$ is the only purely radial one.

Let $\mu(x,y)$ be a charge density supported in the disc
\[
D_R=\{(x,y):\sqrt{x^2+y^2}\le R\}.
\]
Then the electrostatic potential $V$ is harmonic function outside $D_R$, so $\Delta V=0$ in $\mathbb{R}^2\setminus D_R$, while inside $D_R$ one has $\Delta V=2\pi\mu$. (We use a sign convention that is more convenient for our later generalisation.) For $r>R$, $V$ admits the multipole expansion \cite{Joslin1983}
\begin{equation}
    V=p\ln r+\sum_{k=1}^{\infty}\frac{1}{r^k}[a_k\cos(k\theta)+b_k\sin(k\theta)]
    =\operatorname{Re}\left(p\ln z+\sum_{k=1}^{\infty}\frac{c_k}{z^k}\right),
\end{equation}
where $c_k=a_k+i b_k\in\mathbb{C}$ and $p\in\mathbb{R}$. The coefficients are
\begin{equation}\label{eqCoefMultpExpElectroPotential}
    p=\int_{\mathbb{R}^2}\mu\,d xd y,
    \qquad
    c_k=-\frac{1}{k}\int_{\mathbb{R}^2}\mu z^k\,d xd y.
\end{equation}
Taking the imaginary part gives the corresponding expansion for a helicoidal charge density $\mu$:
\begin{equation}\label{EqMultipoleExpHelicCharge}
    h=p\,\theta+\sum_{k=1}^{\infty}\frac{1}{r^k}[b_k\cos(k\theta)-a_k\sin(k\theta)]
    =\operatorname{Im}\left(p\ln z+\sum_{k=1}^{\infty}\frac{c_k}{z^k}\right),
\end{equation}
with the same coefficients as in \eqref{eqCoefMultpExpElectroPotential}.

For a single helicoid of pitch $p_0$ located at $0\in\mathbb{C}$, the harmonic function is
\[
h(r,\theta)=p_0\,\theta=p_0\,\operatorname{Im}\ln z.
\]
Thus, a helicoidal charge $p_0$ at the origin corresponds to the distribution $\mu=p_0\,\delta(z)$, for which $c_k=0$. If instead the charge $p_0$ is located at $z_0\in\mathbb{C}$, then
\[
p_0\ln(z-z_0)=p_0\ln z+p_0\ln\left(1-\frac{z_0}{z}\right)
=p_0\ln z-\sum_{k=1}^{\infty}\frac{p_0}{k}\left(\frac{z_0}{z}\right)^k,
\]
and therefore
\begin{equation}
    h(z)=p_0\,\theta+\sum_{k=1}^{\infty}\frac{p_0}{k}\left(\frac{r_0}{r}\right)^k\sin k(\theta-\theta_0).
\end{equation}
This agrees with \eqref{EqMultipoleExpHelicCharge} for the distribution $\mu=p_0\delta(z-z_0)$, since
\begin{equation}
    p_0=\int p_0\delta(z-z_0),
    \qquad
    c_k=-\frac{1}{k}\int p_0 z^k\,\delta(z-z_0)
    =-\frac{p_0 z_0^k}{k},
\end{equation}
so that
\[
h(z)=\operatorname{Im}\left[p_0\ln z-\sum_{k=1}^{\infty}\frac{p_0}{k}\left(\frac{z_0}{z}\right)^k\right].
\]

More generally, for helicoidal charges $\{p_j\}_{j=1}^N$ located at $\{z_j\}_{j=1}^N\subset\mathbb{C}$, the charge density is
\[
\mu=\sum_{j=1}^N p_j\,\delta(z-z_j).
\]
The multipole coefficients are then
\begin{equation}\label{eq::Coeff2dMultipoleExp}
    p=\sum_{j=1}^N p_j,
    \qquad
    c_k=-\frac{1}{k}\sum_{j=1}^N p_j z_j^k.
\end{equation}
Hence, the associated harmonic function is
\begin{equation}
    h(r,\theta)=\Big(\sum_{j=1}^N p_j\Big)\theta
    +\sum_{k\geq1}\sum_{j=1}^N\frac{p_j}{k}\left(\frac{r_j}{r}\right)^k\sin k(\theta-\theta_j).
\end{equation}
This series converges absolutely for $r>\max_j r_j$, that is, outside the smallest disc containing all charges.

The expansion shows that the fundamental layer of a finite collection of helical motifs behaves asymptotically like a single helicoid whose pitch is equal to the sum of the individual pitches. This description is valid only for finite collections of motifs that are suitably arranged; in particular, infinite configurations or motifs that are not perpendicular to the layers are not captured by this expansion.

\section{Twist grain boundary}
The twist grain boundary (TGB) phase in smectic-A liquid crystals consists of an ordered array of screw dislocations (helical motifs). In this setting, the layers are organised as a linear chain of infinitely many helical motifs. Such structures have been extensively investigated in the study of smectic liquid crystals (see \cite{Santangelo2006, Santangelo2007}). 

We begin by considering a single TGB modelled as a row of screw dislocations aligned along the \( x \)-axis and separated by a uniform distance \( d \). Within the small-slope approximation, the layered structure can be described as level sets of an appropriate phase field 
\[
\phi(x_1,x_2,x_3) = \lambda x_3-\frac{b}{2\pi}\mbox{Im}\ln\sin\left[\frac{\pi(x_1+i x_2)}{d}\right].
\]
Equivalently, individual layers may be identified with the graph of the harmonic function 
\[
h(z) = \frac{b}{2\pi\lambda}\mbox{Im}\ln\sin\left(\frac{\pi z}{d}\right).
\] 
In typical TGB configurations, the helical motifs are assumed to be well separated relative to their pitch. Moreover, TGB admit an exact geometric description in terms of the doubly periodic Scherk’s first minimal surface. It can also be shown that such surfaces may be interpreted as arising from an infinite superposition of helicoid contributions (see \cite{Kamien2001}).

A \( \pi/2 \) TGB phase is formed by an arrangement of twist grain boundaries in which the screw dislocations alternate along the \( \hat{z} \) and \( \hat{x} \) directions, with \( \hat{y} \) taken as the pitch axis, resulting in an overall rotation of \( \pi/2 \). The level sets of a $\pi/2$ TGB phase by summing parallel, alternating TGBs separated by a distance $d$.  This structure can also be derived using infinite product representations of the Jacobi elliptic function \( \operatorname{sn}(u,k) \)(see \cite{Santangelo2007}). In this framework, the exact superposition of screw dislocations for a \( \pi/2 \) TGB configuration can be described in terms of the level sets of the corresponding phase field:
\begin{equation}
\Phi_{\mathrm{TGB}}(\mathbf{x}) = \gamma z - \frac{b}{2\pi}\,\operatorname{Im}\ln \operatorname{sn}(\theta x + i\psi y, k),
\tag{18}
\end{equation}
where $\theta$ and $\psi$ are the necessary scale factors,
\[
K(k)=\int_0^1 dx\left[(1-x^2)(1-k^2x^2)\right]^{-1/2}
\]
is an elliptic period,
\[
iK'(k)=\int_1^{1/k} dx\left[(1-x^2)(1-k^2x^2)\right]^{-1/2}
= iK\bigl(\sqrt{1-k^2}\bigr)
\]
is the other elliptic period, and $k$ is the elliptic modulus. The harmonic function corresponding \( \pi/2 \) TGB phase field is 
\[
h(z) = \frac{b}{2\pi\lambda}\mbox{Im}\ln\operatorname{sn}\left(\theta x + i\psi y, k\right).
\] 

 By analogy with a single grain boundary, which is topologically equivalent to Scherk’s first minimal surface, the resulting structure is likewise related to classical minimal surfaces. In particular, it is topologically equivalent to the triply periodic Schwarz \( D \) surface, another well-known minimal surface (see \cite{Santangelo2007}).

We now consider the construction of an Untwisted
grain boundary (UtGB) configuration as a harmonic graph. The associated harmonic function is given by
\begin{equation}
 h(z)=p\,\mathrm{Im}\!\left(\ln \sin \frac{\pi z}{2d}\right)
      -p\,\mathrm{Im}\!\left(\ln \cos \frac{\pi z}{2d}\right),
\end{equation}
where \( p \) and \( -p \) denote the pitches of the helical motifs forming each dipole pair, separated by a distance \( d \).

\section{Conclusion}
In this work, we established a superposition principle for harmonic immersions with matching complex dilatations. This principle helps examine the non-linear geometric interactions between helical motifs in lamellar structures, which is necessary for understanding inter-motif interactions at distances comparable to their pitch. The harmonic Enneper representation and its associated superposition principle provide a constructive recipe that allows us to examine arbitrary arrangements of helical motifs at any separation and scale, with potential applications in soft matter physics, membrane biophysics, and geometric design.
\section{Acknowledgement}
The author gratefully acknowledges Dr Rahul Kumar Singh for their valuable comments. The author also acknowledges the financial support from the University Grants Commission of India (UGC) under the UGC-JRF scheme (Beneficiary Code: BININ04008604).

\bibliography{Priyank}
\bibliographystyle{ieeetr}

\end{document}